\documentclass[a4paper,reqno]{amsart}

\usepackage{amssymb}
\usepackage{amstext}
\usepackage{amsmath}
\usepackage{amscd}
\usepackage{amsthm}
\usepackage{amsfonts}
\usepackage{enumerate}
\usepackage{graphicx}
\usepackage{latexsym}
\usepackage{mathrsfs}
\usepackage{mathtools}
\usepackage[all]{xy}
\xyoption{all}

\usepackage{pstricks}
\usepackage{lscape}
\usepackage{comment}

\newtheorem{theorem}{Theorem}[section]

\newtheorem{corollary}[theorem]{Corollary}
\newtheorem{lemma}[theorem]{Lemma}
\newtheorem{proposition}[theorem]{Proposition}
\newtheorem{definition-proposition}[theorem]{Definition-Proposition}

\newtheorem{question}[theorem]{Question}

\theoremstyle{definition}
\newtheorem{definition}[theorem]{Definition}

\newtheorem{remark}[theorem]{Remark}
\newtheorem{example}[theorem]{Example}

\newcommand{\Ext}{\operatorname{Ext}\nolimits}

\newcommand{\Hom}{\operatorname{Hom}\nolimits}

\newcommand{\Tr}{\operatorname{Tr}\nolimits}
\renewcommand{\mod}{\mathsf{mod}\hspace{.01in}}

\newcommand{\add}{\mathsf{add}\hspace{.01in}}

\newcommand{\RHom}{\mathbf{R}\strut\kern-.2em\operatorname{Hom}\nolimits}

\numberwithin{equation}{section}

\usepackage{paralist}

\def\Tr{\mathop{\rm Tr}\nolimits}

\def\id{\mathop{\rm id}\nolimits}
\def\pd{\mathop{\rm pd}\nolimits}

\begin{document}
\title{$\tau$-rigid modules over Auslander Algebras}
 \thanks{ The author is supported by NSFC
(Nos.11401488,11571164 and 11671174), NSF of Jiangsu Province (BK
20130983) and Jiangsu Government Scholarship for Overseas Studies
(JS-2014-352). }

\author{Xiaojin Zhang}
\address{School of Mathematics and Statistics, NUIST,
Nanjing, 210044, P. R. China}
\email{xjzhang@nuist.edu.cn}
\maketitle

\begin{abstract} We give a characterization of
$\tau$-rigid modules over Auslander algebras in terms of projective
dimension of modules. Moreover, we show that for an Auslander
algebra $\Lambda$ admitting finite number of non-isomorphic basic
tilting $\Lambda$-modules and tilting $\Lambda^{op}$-modules, if all
indecomposable $\tau$-rigid $\Lambda$-modules of projective
dimension $2$ are of grade $2$, then $\Lambda$ is $\tau$-tilting
finite.
\end{abstract}


\section{Introduction}

Recently Adachi, Iyama and Reiten \cite{AIR} introduced
$\tau$-tilting
 theory to generalize the classical tilting theory in terms of
 mutations. $\tau$-tilting
 theory is close to the silting theory introduced by [AiI] and the cluster tilting
 theory in the sense of \cite{KR, IY, BMRRT}.

  Note that $\tau$-tilting theory depends on $\tau$-rigid
  modules. So it is very interesting to find all $\tau$-rigid modules for a given
 algebra. There are some works on this topic (See \cite{A1, A2, AAC, IJY,IRRT, J, M, HuZh,AnMV,W,Z} and so on).
In particular, Iyama and Zhang \cite{IZ} classified all the support
$\tau$-tilting modules and indecomposable $\tau$-rigid modules for
the Auslander algebra $\Gamma$ of $K[x]/(x^n)$. They showed that the
number of non-isomorphic basic support $\tau$-tilting
$\Gamma$-modules is exactly $(n+1)!$. For an arbitrary Auslander
algebra $\Lambda$, little is known on $\tau$-rigid
$\Lambda$-modules. So a natural question is:

\begin{question}\label{1.0}

How to judge $\tau$-rigid modules over an arbitrary Auslander
algebra?

\end{question}

 Our first goal in this paper is to give a partial answer to this
 question. Throughout this paper all algebras are finite-dimensional
 algebras over a field $K$ and all modules are finitely generated right
 modules.

For an algebra $\Lambda$, denote by $(-)^*$ the functor ${\rm
Hom}_{\Lambda}(-,\Lambda)$. For a $\Lambda$-module $M$, denote by
${\rm pd}_{\Lambda}M$ (resp. ${\rm id}_{\Lambda}M$) the projective
dimension (resp. injective dimension) of $M$. Denote by ${\rm
grade}M$ the grade of $M$. Then we have the following theorem.

\vspace{0.2cm}

\begin{theorem}\label{1.1}(Theorems \ref{2.6} and \ref{2.7}, Corollary \ref{2.b}) Let $\Lambda$ be an Auslander algebra and $M$ a
$\Lambda$-module. Then we have the following:
\begin{enumerate}[\rm(1)]
\setlength{\itemsep}{0pt}
\item  Every simple module $S$ is $\tau$-rigid.
\item  If ${\rm pd}_{\Lambda}M=1$, then $M$ is ($\tau$-)rigid if and only if $\Ext_{\Lambda}^2(N,M)=0$,
where $N=M^{**}/M$.
\item If ${\rm grade}M=2$, then $M$ is $\tau$-rigid if and only if ${\rm
Tr}M$ is $\tau$-rigid with ${\rm pd}_{\Lambda}{\rm Tr}M=1$.
\item  If $\Lambda$ admits a unique simple module $S$ with $\pd_{\Lambda}S=2$, then
\begin{itemize}
\item[\rm (a)] Every indecomposable module $M$ with $\pd_{\Lambda}M=1$ is
($\tau$-)rigid.
\item [\rm (b)] All indecomposable $\tau$-rigid $\Lambda$-modules $N$ with
$\pd_{\Lambda}N=2$ are of grade $2$.
\end{itemize}
\end{enumerate}
\end{theorem}

\vspace{0.2cm}

On the other hand, Demonet, Iyama and Jasso gave a general
description of algebras with finite number of support $\tau$-tilting
modules in \cite{DIJ} where they call the algebras $\tau$-tilting
finite algebras. It is clear that an algebra $\Lambda$ is
$\tau$-tilting finite if and only if so is its opposite algebra
$\Lambda^{op}$. We should remark that an algebra is $\tau$-tilting
finite implies that there are finite number of non-isomorphic basic
tilting $\Lambda$-modules and tilting $\Lambda^{op}$-modules. It is
natural to consider the following question.

\begin{question}\label{1.a}
When is an algebra admitting finite number of basic tilting
$\Lambda$-modules and tilting $\Lambda^{op}$-modules $\tau$-tilting
finite?
\end{question}

 It is obvious that algebras of finite
representation type are both tilting-finite and $\tau$-tilting
finite. However, we need a non-trivial case. Our second goal of this
paper is to give a more general answer to this question whenever
$\Lambda$ is an Auslander algebra. We prove the following theorem in
which the algebra is not necessary to be an Auslander algebra.

\begin{theorem}\label{1.2}(Theorem \ref{2.11}) Let $\Lambda$ be an algebra of global dimension
$2$ admitting finite number of basic tilting $\Lambda$-modules and
tilting $\Lambda^{op}$-modules. If all indecomposable $\tau$-rigid
modules with projective dimension $2$ are of grade $2$, then
$\Lambda$ is $\tau$-tilting finite.
\end{theorem}

The paper is organized as follows:

In Section 2, we recall some preliminaries. In Section 3, we prove
the main results and give some examples to show the main results.

Throughout this paper, all algebras $\Lambda$ are basic connected
finite dimensional algebras over an algebraic closed field $K$ and
all $\Lambda$-modules are finitely generated right modules. Denote
by $\mod\Lambda$ the category of finitely generated right
$\Lambda$-modules. For $M\in\mod\Lambda$, denote by ${\rm add}M$ the
subcategory of direct summands of finite direct sum of $M$. We use
${\rm Tr}M$ to denote the Auslander transpose of $M$. Denote by
$\tau$ the $AR$-translation and denote by $|M|$ the number of
non-isomorphic indecomposable direct summands of $M$.

\vspace{0.2cm}

{\bf Acknowledgement} Part of this work was done when the author
visited Nagoya University in the year 2015. The author
 would like to thank Osamu Iyama, Laurent Demonet, Takahide Adachi,
Yuta Kimura, Yuya Mizuno and Yingying Zhang for useful discussion
and kind help. The author also wants to thank Osamu Iyama for
hospitality during his stay in Nagoya.

\section{Preliminaries}

In this section we recall some basic preliminaries for later use.
For an algebra $\Lambda$, denote by ${\rm gl.dim}\Lambda$ the global
dimension of $\Lambda$. We begin with the definition of Auslander
algebras.

\vspace{0.2cm}

\begin{definition}\label{2.1} An algebra $R$ is called an {\it Auslander algebra}  if ${\rm
gl.dim}R\leq2$ and $I_i(R)$ is projective for $i=0,1$, where
$I_i(R)$ is the $(i+1)$-th term in a minimal injective resolution of
$R$.
\end{definition}

\vspace{0.2cm}

Let $R$ be a representation-finite algebra  and $A$ an additive
generator of $\mod R$. Auslander proved that there is a one to one
correspondence between representation-finite algebras and Auslander
algebras via $R\mapsto{\rm End}_{R}(A)$. In this case, we call
$\mathrm{End}_{R} \left( A \right)$ the {\it Auslander algebra of
R}. Furthermore, for $X \in \mod R$ we denote by $ P_{X} =
\mathrm{Hom}_{R} \left( A,X \right)$ and $S_{X} = P_{X} /
\mathrm{rad}P_{X}$. The following statement [AuRS] is essential in
the proof of the main result.

\begin{proposition}\label{2.2} Let $X$ be an indecomposable $R$-module.
Then
\begin{itemize}
\item[\rm(1)] $\mathrm{pd}_{\Lambda} S_{X} \leq 1$ if and only if $X$
 is projective, and
$0 \rightarrow P_{\mathrm{rad}X} \rightarrow P_{X} \rightarrow S_{X}
 \rightarrow 0$
 is a minimal projective resolution of $S_{X}$.

\item[\rm(2)] $\mathrm{pd}_{\Lambda}S_{X} = 2$ if and only if
       $X$ is not projective, and the almost split sequence
 $0 \rightarrow \tau X \rightarrow E \rightarrow X \rightarrow 0$ gives
      a minimal projective resolution $0 \rightarrow P_{\tau X} \rightarrow
 P_{E} \rightarrow P_{X} \rightarrow S_{X} \rightarrow 0$ of $S_{X}$.
 \end{itemize}
 \end{proposition}

For a positive integer $k$, an algebra $\Lambda$ is called
Auslander's $k$-Gorenstein if $\pd_{\Lambda}I_j(\Lambda)\leq j$ for
$0\leq j\leq k-1$. For a $\Lambda$-module $M$ and a positive integer
$i$, we call ${\rm grade}M\geq i$ if $\Ext_{\Lambda}^{
j}(M,\Lambda)=0$ for $0\leq j\leq i-1$. We need the following
result.

\begin{lemma}\label{2.3} Let $\Lambda$ be an Auslander algebra and $T
\in \mod\Lambda$.  For $j=1,2$,
\begin{itemize}

\item[\rm(1)] The subcategory $\{M|{\rm grade}M\geq j\}$ is closed under
submodules and factor modules.

\item[\rm(2)] Every simple $\Lambda$-module $S$ is either of grade $0$ or of
grade $2$.

\item[\rm(3)] ${\rm grade}\  \Ext_{\Lambda}^{j}(T,\Lambda)\geq
2$.
\item[\rm(4)] The projective dimension of any composition factor of
$\mathrm{Ext}_{\Lambda}^{2} \left( T, \Lambda \right)$ is $2$.
\end{itemize}
\end{lemma}

\begin{proof}
(1) is a straight result of \cite[Proposition 2.4]{I}.

(2) follows from the fact $\Ext_{\Lambda}^{i}(S,\Lambda)\simeq
\Hom_{\Lambda}(S,I_i(\Lambda))$ and $\Lambda$ is an Auslander
algebra.

(3) By the definition of Auslander algebra, $\Lambda$ is Auslander's
$2$-Gorenstein. Then by \cite{FGR} $\Lambda$ is Auslander's
$k$-Gorenstein if and only if for each submodule $X$ of
$\Ext_{\Lambda}^{i}(T,\Lambda)$ with $T$ in $\mod \Lambda$ and
$i\leq k$, we have ${\rm grade}X\geq i$. Then we have ${\rm grade}\
\Ext_{\Lambda}^{j}(T,\Lambda)\geq j$ for $j=1,2$. By (1) every
composition factor $S$ of $\Ext_{\Lambda}^1(T,\Lambda)$ has grade at
least $1$, and hence $2$ by (2). Then by an induction on the length
of $\Ext_{\Lambda}^1(T,\Lambda)$, we get ${\rm grade}
\Ext_{\Lambda}^{1}(T,\Lambda)\geq 2$.

 (4) is a direct result
of (1) and (3).
\end{proof}

In the following we recall some basic properties of $\tau$-rigid
 modules. We start
 with the following definition \cite{AIR}.

\begin{definition}\label{2.8} We call $M \in \mod\Lambda$ {\it
$\tau$-rigid} if ${\rm Hom}_{\Lambda}(M, \tau M) = 0$. In addition,
 $M$ is called {\it $\tau$-tilting} if $M$ is
$\tau$-rigid and $|M| = |\Lambda|$. Moreover, $M$ is called {\it
support $\tau$-tilting} if there exists an idempotent $e$ of
$\Lambda$ such that $M$ is a $\tau$-tilting $\Lambda/(e)$-module.
\end{definition}

It is clear that any $\tau$-rigid $\Lambda$-module $M$ is rigid,
that is, $\Ext_{\Lambda}^1(M,M)=0$. In general the converse is not
true. But if ${\rm pd}_{\Lambda}M=1$, then $M$ is $\tau$-rigid if
and only if $M$ is rigid. Recall that a $\Lambda$-module $T$ is
called a {\it (classical) tilting module} if $T$ satisfies (1)
$\mathrm{pd}_{\Lambda} T \leq 1$, (2)
$\mathrm{Ext}_{\Lambda}^{1}(T,T) = 0$ and ($3$) $|T|=|\Lambda|$. It
is showed in [AIR] that a tilting $\Lambda$-module is exactly a
faithful support $\tau$-tilting $\Lambda$-module.

To judge $\tau$-rigid modules of projective dimension $2$ over
Auslander algebras, we also need the following lemma in [AIR].

\begin{lemma}\label{2.9} Let $\Lambda$ be an algebra and $M$ a
$\Lambda$-module without projective direct summands. Then $M$ is
$\tau$-rigid in $\mod\Lambda$ if and only if ${\rm Tr}M$ is
$\tau$-rigid in $\mod\Lambda^{op}$.
\end{lemma}

Recall that a morphism $f:M\rightarrow N$ is called {\it right
minimal}  (resp. {\it left minimal}) if $fg=f$ (resp. $gf=f$)
implies that $g$ is an isomorphism, where $g$ is a homomorphism of
the form $M\rightarrow M$ (resp. $N\rightarrow N$). The following
properties of right minimal (resp. left minimal) morphisms in
\cite{HuZ} are useful for the proof of the main results.
\begin{lemma}\label{2.a} Let $0\rightarrow A\stackrel{g}{\rightarrow}
B\stackrel{f}{\rightarrow} C\rightarrow0$ be a non-split exact
sequence in $\mod\Lambda$ with $B$ projective-injective. Then the
following are equivalent:
\begin{enumerate}[\rm(1)]
\item A is indecomposable and $g$ is left minimal.
\item C is indecomposable and $f$ is right minimal.

\end{enumerate}

\end{lemma}

\section{Main results}

In this section we give the main results of this paper and some
examples to show the main results. Throughout this section,
$\Lambda={\rm End}_{R}A$ is the Auslander algebra of a
representation-finite algebra $R$ with an additive generator $A$.

\vspace{0.2cm}

It is showed by Igusa \cite{Ig} that $S$ is rigid for any simple
module $S$ over an algebra $\Gamma$ of finite global dimension.
However, we give a new direct proof for the rigidness of simple
modules whenever $\Gamma$ is an Auslander algebra.

\begin{proposition}\label{2.4}  Let $\Lambda$ be an Auslander algebra and $S$ a
simple $\Lambda$-module. Then $\Ext_{\Lambda}^{1}(S,S)=0$.
\end{proposition}

\begin{proof}

For a simple $\Lambda$-module $S$, we show the assertion by using
the projective dimension of $S$.

If ${\rm pd}_{\Lambda}S=0$, there is nothing to show.

If ${\rm pd}_{\Lambda}S=1$, then we can get a minimal projective
resolution $0\rightarrow P_1(S)\rightarrow P_0(S)\rightarrow
S\rightarrow 0$. Then the length of $P_1(S)$ is smaller than that of
$P_0(S)$, and hence ${\rm Hom}_{\Lambda}(P_1(S),S)=0$. So one gets
$\Ext_{\Lambda}^1(S,S)\simeq {\rm Hom}_{\Lambda}(P_1(S),S)=0$.

If ${\rm pd}_{\Lambda}S=2$, then by Proposition 2.2, there is an
$AR$-sequence $0\rightarrow \tau X\rightarrow E\rightarrow
X\rightarrow 0$ in $\mod R$ such that $0\rightarrow {\rm
Hom}_{R}(A,\tau X)\rightarrow {\rm Hom}_{R}(A,E)\rightarrow {\rm
Hom}_{R}(A,X)\rightarrow S\rightarrow 0$ is a minimal projective
resolution of $S$. On the contrary, suppose that ${\rm
Ext}_{\Lambda}^{1}(S,S)\not=0$, then we get that ${\rm
Hom}_{\Lambda}(P_1(S),S)\simeq \Ext_{\Lambda}^{1}(S,S)\neq 0$. So
$P_0(S)= {\rm Hom_{R}(A,X)}$ is a direct summand of $P_1(S)={\rm
Hom}_{\Lambda}(A,E)$. Note that the functor ${\rm
Hom}_{\Lambda}(A,-)$ induces an equivalence from ${\rm add} A$ to
${\rm add}\Lambda$, then $X$ is a direct summand of $E$. Since
$E\rightarrow X$ is right almost split, then we get an irreducible
morphism $f:X\rightarrow X$ by \cite[IV, Theorem 1.10(b)]{AsSS}, a
contradiction.
\end{proof}

Denote by $(-)^*$ the functor ${\rm Hom}_{\Lambda}(-,\Lambda)$, then
we have the following lemma \cite{IZ} with a different shorter
proof.

\begin{lemma}\label{2.5} Let $\Lambda$ be an Auslander algebra, and let
$M$ be a $\Lambda$-module with $\pd_{\Lambda}M \leq 1$. Then the
canonical map $M \stackrel{\varphi_{M}}{\rightarrow}M^{\ast \ast}$
is injective, and the projective dimension of any composition factor
of $M^{\ast \ast} / M$ is $2$.
\end{lemma}

\begin{proof}
By \cite{AuB}, we get an exact sequence $0\rightarrow {\rm
Ext}_{\Lambda^{op}}^1({\rm Tr} M, \Lambda)\rightarrow M\rightarrow
M^{**}\rightarrow \Ext_{\Lambda^{op}}^2({\rm Tr} M,
\Lambda)\rightarrow 0 $. To show the former assertion, it suffices
to show that $\mathrm{Ext}_{\Lambda^{op}}^{1} \left( \mathrm{Tr}M ,
\Lambda \right) = 0$. In the following we show ${\rm grade}\Tr M=2$.
Since
 ${\rm pd}_{\Lambda}M\leq 1$, then one gets ${\rm Tr} M\simeq {\rm
 Ext}_{\Lambda}^1(M,\Lambda)$ and hence by Lemma 2.3(3), ${\rm grade\ Tr }M\geq
 2$ holds, and hence $\Ext_{\Lambda^{op}}^1(\Tr M,\Lambda)=0$  We get the desired injection.
 Then by Lemma 2.3(2), the later assertion holds.
\end{proof}

 Now we are in a position to state the following main result on the
 ($\tau$-)rigidness of modules with projective dimension $1$.

\begin{theorem}\label{2.6} Let $\Lambda$ be an Auslander algebra and
$M$ a $\Lambda$-module with ${\rm pd}_{\Lambda}M=1$.
 Then $\Ext_{\Lambda}^{1}(M,M)=0$ if and only if ${\rm
Ext}_{\Lambda}^{2}(N,M)=0$ holds for $N=M^{**}/M$.
\end{theorem}

\begin{proof}
We show the assertion step by step.

 (1) For any $M\in {\rm
 mod}\Lambda$, $M^*$ is projective. Here we only
 need the condition ${\rm gl.dim}\Lambda=2.$

 Let $P_{1}(M) \rightarrow P_{0}(M) \rightarrow M \rightarrow
0$ be a projective resolution of $M$. Applying the functor $\left( -
\right)^{\ast}$, we get an exact sequence
 $ 0 \rightarrow M^{\ast} \rightarrow P_{0}(M)^{\ast}
  \rightarrow P_{1}(M)^{\ast}$. Since
$\mathrm{gl.dim} \Lambda \leq 2$, one gets that $M^{\ast}$ is a
projective $\Lambda^{op}$-module. Thus $M^{\ast \ast}$ is a
projective $\Lambda$-module.

(2) $\Ext_{\Lambda}^1(M,M)\simeq \Ext_{\Lambda}^2(M^{**}/M,M)$
holds.

 By Lemma \ref{2.5}, we get the exact
sequence $0\rightarrow M\rightarrow M^{**}\rightarrow {\rm
Ext}_{\Lambda^{op}}^{2}({\rm Tr}M, \Lambda)(=M^{**}/M)\rightarrow0$.
Applying the functor ${\rm Hom}_{\Lambda}(-,M)$ to the exact
sequence, we get the desired isomorphism since $M^{**}$ is
projective by (1).
\end{proof}

Immediately, we have the following corollary.

\begin{corollary}\label{3.4} Let $\Lambda$ be an Auslander algebra and
$M$ a $\Lambda$-module with ${\rm pd}_{\Lambda}M=1$.
\begin{enumerate}[\rm(1)]
\item If ${\rm id}_{\Lambda}M=1$, then $\Ext_{\Lambda}^{1}(M,M)=0$ holds.
\item If $\Ext_{\Lambda}^{2}(S',M)=0$ holds for any composition factor
$S'$ of $M^{**}/M$, then $\Ext_{\Lambda}^{1}(M,M)=0$ holds.
\end{enumerate}
 \end{corollary}
\begin{proof}
(1) follows from Theorem \ref{2.6} directly. By induction on the
length of $M^{**}/M$, one can get the assertion (2).
\end{proof}

\begin{remark} We should remark that the converse of Corollary
\ref{3.4} are not true in general (see Example \ref{3.9}(5)).

\end{remark}

Denote by ${\rm i}\tau$-${\rm rig}\Lambda$ the set of isomorphism
classes of indecomposable $\tau$-rigid $\Lambda$-modules. Similarly,
one can define ${\rm i}\tau$-${\rm rig}\Lambda^{op}$. Denote by
 $\mathcal{G}$ the subset of ${\rm i}\tau$-${\rm rig}\Lambda$
consisting of isomorphism classes of $\tau$-rigid modules of grade
$2$ and denote by $\mathcal{S}$ the subset of ${\rm i}\tau$-${\rm
rig}\Lambda^{op}$ consisting of isomorphism classes of
non-projective $\tau$-rigid submodules of $\add\Lambda^{op}$.
 To judge $\tau$-rigid modules of projective dimension 2 over Auslander algebras, we need the following proposition.
\begin{proposition}\label{2.10} Let $\Lambda$ be an algebra of global
dimension 2. There is a bijection between $\mathcal{G}$ and
$\mathcal{S}$ via ${\rm Tr}: M\mapsto {\rm Tr M}$.
\end{proposition}

\begin{proof}

By Lemma \ref{2.9} $M$ is $\tau$-rigid if and only if $\Tr M$ is
$\tau$-rigid. Now it suffices to show that (a) $M\in\mathcal{G}$
implies that $\Tr M\in\mathcal{S}$ and (b) $M\in\mathcal{S}$ implies
that $\Tr M\in\mathcal{G}$.

(a) Since $M\in\mathcal{G}$, take the following minimal projective
resolution of $M$: $\cdots\rightarrow P_1(M)\rightarrow
P_0(M)\rightarrow M\rightarrow 0$. Applying the functor
$(-)^{\ast}$, we get an exact sequence
\begin{equation}\label {2.11.1}
0=M^{\ast}\rightarrow P_0(M)^{\ast}\rightarrow
P_1(M)^{\ast}\rightarrow \Tr M\rightarrow 0,
\end{equation}
which is a minimal projective resolution of $\Tr M $. Then
$\pd_{\Lambda}\Tr M=1$. On the other hand, since ${\rm grade}M=2$,
one gets the following sequences

\begin{equation}\label {2.11.2}0=M^{\ast}\rightarrow
P_0(M)^{\ast}\rightarrow \Omega^1M^{\ast} \rightarrow
\Ext_{\Lambda}^{1}(M,\Lambda)=0.
\end{equation}
and
\begin{equation}\label {2.11.3}
 0\rightarrow \Omega^1M^{\ast}\rightarrow
P_1(M)^{\ast}\rightarrow P_2(M)^{\ast}
\end{equation}

 Comparing exact sequences \ref{2.11.1} with \ref{2.11.2} and
 \ref{2.11.3}, one gets that $\Tr M$ is a submodule of
 $P_2(M)^{\ast}$.

(b) Since $M\in \mathcal{S}$ is non-projective and ${\rm
gl.dim}\Lambda=2$, then $\pd_{\Lambda}M=1$. Take a minimal
projective resolution of $M$: $0\rightarrow P_1(M)\rightarrow
P_0(M)\rightarrow M\rightarrow0$. Applying $(-)^{\ast}$, we get the
following exact sequence $0\rightarrow M^{\ast} \rightarrow
P_0(M)^{\ast}\rightarrow P_1(M)^{\ast}\rightarrow \Tr M\rightarrow 0
$. Note that $\Tr$ is a duality and $\pd_{\Lambda}M=1$, one gets
that $\Hom_{\Lambda^{op}}(\Tr M, \Lambda)=0$. Since $M$ can be
embedded into a projective module, then $M$ is torsionless, that is
$M\rightarrow M^{\ast\ast}$ is injective. By \cite{AuB} there is an
exact sequence $0\rightarrow\Ext_{\Lambda^{op}}^{1}(\Tr
M,\Lambda)\rightarrow M\rightarrow
M^{\ast\ast}\rightarrow\Ext_{\Lambda^{op}}^{2}(\Tr
M,\Lambda)\rightarrow0$ which implies that
$\Ext_{\Lambda^{op}}^{1}(\Tr M,\Lambda)=0$. Then {\rm grade}$\Tr
M=2$.
\end{proof}
As a corollary, we get the following.

\begin{corollary}\label{2.b}
Let $\Lambda$ be an Auslander algebra and $M\in\mod\Lambda$. If $M$
is of {\rm grade} $2$, then $M$ is $\tau$-rigid if and only if $\Tr
M$ is $\tau$-rigid with $\pd_{\Lambda}\Tr M=1$ in
$\mod\Lambda^{op}$.
\end{corollary}

\begin{proof}
By Proposition \ref{2.10}, it is enough to show that
$\pd_{\Lambda}M=1$ if and only if $M$ can be embedded into a
projective module. Since ${\rm gl.dim}\Lambda=2$, one gets that $M$
can be embedded into a projective module implies that
$\pd_{\Lambda}M=1$. The converse follows from Lemma \ref{2.5}.
\end{proof}

\vspace{0.2cm}

Recall that from \cite{DIJ} that an algebra $\Lambda$ is called
{$\tau$-tilting finite} if there are finite number of non-isomorphic
indecomposable $\tau$-rigid modules in $\mod\Lambda$. It is clear
that a $\tau$-tilting finite algebra admits finite number of tilting
$\Lambda$-modules and tilting $\Lambda^{op}$-modules. To find a way
from two-sided tilting finite to $\tau$-tilting finite, we have the
following.

\begin{theorem}\label{2.11} Let $\Lambda$ be an algebra of global
dimension 2 admitting finite number of basic tilting
$\Lambda$-modules and tilting $\Lambda^{op}$-modules. If all
indecomposable $\tau$-rigid modules $M$ with ${\rm pd}_{\Lambda}M=2$
are of ${\rm grade}\ 2$, then $\Lambda$ is $\tau$-tilting finite.
\end{theorem}

\begin{proof}

By the assumption, there are finite number of tilting modules which
implies that there are finite number of indecomposable $\tau$-rigid
$\Lambda$-modules and $\Lambda^{op}$-modules of projective dimension
less than or equal to $1$. Then by Proposition \ref{2.10}, the
number of indecomposable $\tau$-rigid $\Lambda$-module of grade $2$
is equal to the number of indecomposable non-projective $\tau$-rigid
submodules $N$ of $\Lambda^{op}$. Since ${\rm gl.dim}\Lambda=2$, we
get that $\pd_{\Lambda}N=1$, and hence the number of this class of
modules is finite. Note that all indecomposable $\tau$-rigid
$\Lambda$-modules with projective dimension 2 are of grade 2, then
the number of indecomposable $\tau$-rigid modules with projective
dimension $2$ is finite by Proposition \ref{2.10}.
\end{proof}

Immediately, we have the following corollary which confirms the
$\tau$-tilting finiteness of the the Auslander algebra of
$K[x]/(x^n)$ showed in \cite{IZ}.

\begin{corollary}\label{2.12} Let $\Lambda$ be an Auslander algebra
admitting finite number of basic tilting $\Lambda$-modules and
tilting $\Lambda^{op}$-modules. If all indecomposable $\tau$-rigid
modules $M$ with ${\rm pd}_{\Lambda}M=2$ are of ${\rm grade}\ 2$,
then $\Lambda$ is $\tau$-tilting finite.
\end{corollary}

For a module $M$, denote by ${\rm rad}M$ and ${\rm soc}M$ the
radical and the socle of $M$, respectively. Now we give the
following classification of Auslander algebras admitting a unique
simple module of projective dimension 2 which gives a support to
Theorem \ref{2.6} and Corollary \ref{2.12}.

\begin{theorem}\label{2.7} Let $\Lambda$ be an Auslander algebra. If $\Lambda$ admits a unique simple $\Lambda$-module
 $S$ with ${\rm pd}_{\Lambda}S=2$, then
\begin{itemize}
\item[\rm(1)] $\Lambda$ is either the Auslander algebra of the path algebra $R=KQ$ with $Q:1\rightarrow 2$ or the
Auslander algebra of the Nakayama local algebra $R$ of radical
square zero.
\item[\rm(2)] Every indecomposable $\Lambda$-module $M$ with ${\rm pd}_{\Lambda}M\leq 1$
is rigid, and hence $\tau$-rigid.
\item[\rm(3)] All indecomposable $\tau$-rigid $\Lambda$-modules $N$ with
$\pd_{\Lambda}N=2$ are of grade $2$.
\end{itemize}
\end{theorem}

\begin{proof}
Since (2) and (3) follow from (1) easily, we only show (1). By
Proposition 2.2, there is a unique non-projective indecomposable
$R$-module $X$ such that the $AR$-sequence $0\rightarrow \tau
X\rightarrow E\rightarrow X\rightarrow 0$ in $\mod R$ induces a
minimal projective resolution of $S$: $0\rightarrow {\rm
Hom}_{R}(A,\tau X)\rightarrow {\rm Hom}_{R}(A,E)\rightarrow {\rm
Hom}_{R}(A,X)\rightarrow S\rightarrow 0$. Then all indecomposable
modules are projective except $X$. We claim that $X$ should be
simple. Otherwise, there would be a simple factor module $Y$ of $X$
such that $Y\not\simeq X$. By the proof above $Y$ would be
projective and hence $X\simeq Y$ is projective, a contradiction. Now
we divide the proof in two parts.

(a) If $X$ is not injective, then all indecomposable injective
$R$-modules are projective, and hence $R$ is self-injective. So we
get that $R$ is local with a unique simple module $X$. Otherwise,
there would be a simple projective-injective $R$-module. One gets a
contradiction since $R$ is basic and connected. Taking a minimal
projective resolution of $X$, we get the following exact sequence
$0\rightarrow \Omega^1 X\rightarrow P_0(X)(=R)\rightarrow
X\rightarrow 0$. By Lemma 2.6, $\Omega^1 X$ is indecomposable
non-projective, and hence $\Omega^1 X\simeq X$. Then ${\rm rad}^2
R=0$ holds. By \cite[IV , Proposition 2.16]{AuRS}, $R$ is a Nakayama
algebra.

(b) If $X$ is injective, then $X\not\simeq {\rm soc} P$ for any
indecomposable projective $R$-module. Hence the injective envelope
$I^0(R)$ is projective, that is, $R$ is Auslander's $1$-Gorenstein
\cite{FGR}.  Then $P_0(X)$ is projective-injective since $X$ is
injective. Taking a part of minimal projective resolution of $X$:
$0\rightarrow\Omega^{1}X\rightarrow P_0(X)\rightarrow
X\rightarrow0$, one gets that $\Omega^{1}X$ is indecomposable and
projective by Lemma \ref{2.a}. Then we conclude that $R$ is a
hereditary algebra.

In the following we show $R$ is a Nakayama algebra. One can show
that $P_0(X)$ is the unique projective-injective module in $\mod R$
since $R$ is a basic connected hereditary algebra. Then every
indecomposable projective $R$-module is contained in $P_0(X)$ and
admits a unique composition series. By \cite{FGR}, $R^{op}$ is also
Auslander's $1$-Gorenstein. Similarly, every indecomposable
projective $R^{op}$-module admits a unique composition series. So
$R$ is a Nakayama algebra. By \cite[V, Theorem 3.2]{AsSS} and the
fact all indecomposable $R$-modules are projective except one, we
get that $R=KQ$ with $Q:1\rightarrow 2$.
\end{proof}

At the end of this paper we give another two examples to show our
main results.

\begin{example}\label{3.9}

Let $\Lambda$ be the Auslander algebra of $K[x]/(x^n)$. Then we have
the following:
\begin{enumerate}[\rm(1)]
\item  $\Lambda$ is given by \[\xymatrix{
1\ar@<2pt>[r]^{a_1}&2\ar@<2pt>[r]^{a_2}\ar@<2pt>[l]^{b_2}&3\ar@<2pt>[r]^{a_3}\ar@<2pt>[l]^{b_3}&\cdots\ar@<2pt>[r]^{a_{n-2}}\ar@<2pt>[l]^{b_4}&n-1\ar@<2pt>[r]^{a_{n-1}}\ar@<2pt>[l]^{b_{n-1}}&n\ar@<2pt>[l]^{b_n}
}\] with relations $a_{1} b_{2}= 0$ and $a_{i} b_{i+1} =b_{i}
a_{i-1}$ for any $2 \leq i \leq n-1$. $\Lambda$ is of infinite
representation type if $n\geq 5$.

\item All indecomposable module $M$ with
$\pd_{\Lambda}M=1=\id_{\Lambda}M$ are direct summands of tilting
modules, and hence $\tau$-rigid.

\item All indecomposable $\tau$-rigid modules of projective
dimension $2$ are of grade $2$. (See \cite{IZ} for details)

\item The number of tilting $\Lambda$-modules (resp.
$\Lambda^{op}$-modules) is $n!$ (\cite{IZ,T}). By Theorem
\ref{2.11}, $\Lambda$ is $\tau$-tilting finite.

\item If $n=4$, then the indecomposable module $M=\begin{smallmatrix}
&2&&4\\&&3\\&&&4\end{smallmatrix}$ is ($\tau$-)rigid with
$\pd_{\Lambda}M=1$ and $\id_{\Lambda}M=2$ and $M^{**}=\begin{smallmatrix} &2\\ 1&&3\\
&2&&4\\&&3\\&&&4\end{smallmatrix}$. But
$\Ext_{\Lambda}^2(S(2),M)\neq0$.

\end{enumerate}
\end{example}

We should remark that there does exist an Auslander algebra
$\Lambda$ such that an indecomposable $\tau$-rigid $\Lambda$-module
with projective dimension $2$ does not necessarily have grade $2$.
\begin{example}\label{3.10}
Let $\Lambda$ be the Auslander algebra of $KQ$ with
$Q:1\stackrel{a_1}{\rightarrow}2\stackrel{a_2}{\rightarrow}3$. Then
\begin{enumerate}[\rm(1)]
\item $\Lambda$ is given by the following quiver $Q':$
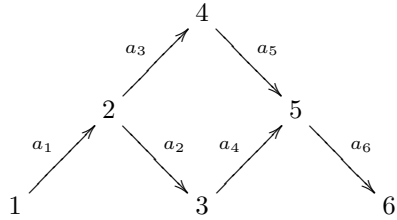
\[\xymatrix{\ &\ &4\ar[dr]^{a_5}&\ &\ \\
\ &2\ar[ur]^{a_3}\ar[dr]^{a_2}&\ &5\ar[dr]^{a_6}&\ \\
1\ar[ur]^{a_1}&\ &3\ar[ur]^{a_4}&\ &6}\] with relations $a_2a_1=0$,
$a_5a_3=a_4a_2$ and $a_6a_4=0$.
\item All indecomposable modules are $\tau$-rigid.
\item The indecomposable module $M=\begin{smallmatrix} &2 \\
3&&4\end{smallmatrix}$ is of projective dimension $2$, but it is not
of grade $2$ since $\pd_{\Lambda}S(4)=1$.
\end{enumerate}

\end{example}

\end{document}